\documentclass[12pt]{article}
\usepackage{amsfonts}
\usepackage{graphicx}
\usepackage{amsmath}

\numberwithin{equation}{section}

\newtheorem{theorem}{Theorem}[section]

\newtheorem{corollary}[theorem]{Corollary}

\newtheorem{lemma}[theorem]{Lemma}

\newtheorem{proposition}[theorem]{Proposition}
\newtheorem{remark}[theorem]{Remark}

\newenvironment{proof}[1][Proof]{\textbf{#1.} }{\ \rule{0.5em}{0.5em}}

\begin{document}

\title{Skorohod Equation and Reflected Backward\\
Stochastic Differential Equations}
\date{}
\author{\emph{{\small {\textsc{By Zhongmin QIAN\ and Mingyu XU}}}} \\
\emph{{\small {Oxford University and Chinese Academy Sciences}}}}
\maketitle

\leftskip1truecm \rightskip1truecm \noindent {\textbf{Abstract.} By using
the Skorohod equation we derive an
iteration procedure which allows us to solve a class of reflected backward
stochastic differential equations with non-linear resistance induced by the
reflected local time. In particular, we present a new method to study the
reflected BSDE proposed first by El Karoui et al. \cite{MR1434123}.}

\leftskip0truecm \rightskip0truecm

\vskip6truecm


\vskip0.5truecm


\newpage

\section{Introduction}

Backward stochastic differential equations (BSDEs in short), in the linear
case and a special non-linear case, were first introduced by Bismut \cite%
{MR0453161} in the study of the stochastic maximal principle. General
non-linear BSDEs were first considered by Pardoux and Peng \cite{MR1037747},
who proved the existence and uniqueness of adapted solutions, under smooth
square-integrability assumptions on the coefficient and the terminal data,
and when the coefficient $f(t,\omega ,y,z)$ is Lipschitz in $(y,z)$
uniformly in $(t,\omega )$. El Karoui, Kapoudjian, Pardoux, Peng and Quenez
introduced the notion of reflected BSDE {\cite{MR1434123},} with one
continuous lower barrier. More precisely, a solution for such equation
associated with a coefficient $f$, a terminal value $\xi $, a continuous
barrier $S$ which is modelled by a semimartingale, is a triple $(Y,Z,K)$ of
adapted stochastic processes in $\mathbb{R}^{1+d+1}$, which satisfies the
following stochastic integral equation 
\begin{equation}
Y_{t}=\xi +\int_{t}^{T}f(s,Y_{s},Z_{s})ds+K_{T}-K_{t}-\int_{t}^{T}Z_{s}dB_{s}
\label{RBSDEone}
\end{equation}%
for any $0\leq t\leq T$ and $Y\geq S$ a.s., $K$ is non-decreasing
continuous, where $B$ is a $d$-dimensional Brownian motion. The role of $K$
is to push upward the process $Y$ in a minimal way, while to keep it above $S
$. In this sense it satisfies 
\begin{equation}
\int_{0}^{T}(Y_{s}-S_{s})dK_{s}=0\text{.}  \label{RBSDEob}
\end{equation}%
In order to prove the existence and uniqueness of the solution, they used
first a Picard-type iterative procedure, which requires at each step the
solution of an optimal stopping problem. The second approximation is
constructed by penalization of the constraint. From then on, many researches
have been done to relax the assumptions of reflected BSDE in {\cite%
{MR1434123}} based on these two main methods. In \cite{MR1467440}, the case
that the coefficient $f$ are not Lipschitz function with only linear growth
has been considered. In \cite{MR2135157} monotonicity condition is used
instead of the Lipschitz condition, and in \cite{MR1944142}, \cite{LX07}, 
\cite{MR2418253}, these authors studied RBSDE with quadratic increasing
condition. In another direction, different barrier conditions have been
studied, for example the barrier $S$ is not continuous, but only right
continuous with left limits, (cf. \cite{MR2185610}), or is in more general
case, cf. \cite{MR2139035}.

Recently a new type of reflected BSDEs has been introduced by Bank and El
Karoui \cite{MR2044673} by a variation of Skorohod's obstacle problem, which
is named as variant reflected BSDE, which has been generalized by Ma and
Wang in \cite{MR2511615}. The formulation of such equation with an optional
process $X$ (as an upper barrier) 
\begin{equation*}
Y_{t}=X_{T}+\int_{t}^{T}f(s,Y_{s},Z_{s},A_{s})ds-\int_{t}^{T}Z_{s}dB_{s}%
\text{ and }Y\leq X\text{,}
\end{equation*}%
where $A$ is an increasing process, with $A_{0-}=-\infty $, and the flat-off
condition holds $\int_{t}^{T}|Y_{s}-X_{s}|dA_{s}=0$. Here the increasing
process $A$ does not directly act on $Y$ to push the solution downwards such
that $Y_{t}\leq X_{t}$, instead it acts through the generator $f$ which is
decreasing in $A$, like a 'density' of reflecting force. In \cite{MR2511615}%
, it has been proved that the solution in a small-time duration, under some
extra conditions, exists and is unique.

However we still do not know much about the increasing process $K$. In all
these papers mentioned above, there are few results considering the
increasing process $K$. In this paper, we will use the Skorohod equation to
represent the increasing process $K$ in terms of the solution $Y$ and $Z$.
This representation is \textquotedblright explicit\textquotedblright\
showing how the force $K$ pushes the solution $Y$ according to the barrier $%
S $ and is given by 
\begin{eqnarray*}
K_{t} &=&\max \left[ 0,\max_{0\leq s\leq T}\left\{ -\left( \xi
+\int_{s}^{T}f(r,Y_{r},Z_{r})dr-S_{s}-\int_{s}^{T}Z_{r}dB_{r}\right)
\right\} \right] \\
&&-\max \left[ 0,\max_{t\leq s\leq T}\left\{ -\left( \xi
+\int_{s}^{T}f(r,Y_{r},Z_{r})dr-S_{s}-\int_{s}^{T}Z_{r}dB_{r}\right)
\right\} \right] \text{.}
\end{eqnarray*}
Together with the theory of optional dual projections (for details about the
general theory, see for example \cite{HWY}), we construct a new Picard
iteration based on this formula and prove that if $(Y,Z,K)$ is the fixed
point, then it is the solution of reflected BSDE.

With this approach we are able to consider the following type of reflected
BSDE 
\begin{equation}
\text{ \ \ \ }Y_{t}=\xi
+\int_{t}^{T}f(s,Y_{s},Z_{s},K_{s})ds+K_{T}-K_{t}-\int_{t}^{T}Z_{s}dB_{s}
\label{rbsde-0}
\end{equation}%
for \ $t\leq T$, subject to the constrain that 
\begin{equation*}
Y_{t}\geq S_{t}\text{ \ and \ }\int_{0}^{T}(Y_{t}-S_{t})dK_{t}=0\text{.\ \ \
\ }
\end{equation*}%
Here the force is still given by an increasing process $K$, which satisfies
the flat-off condition, but $K$ also appears in the driver $f$ as a
resistance force. If $f$ is decreasing in $K$, then we get an extra force
from the Lebesgue integral. If $f\,$\ is increasing in $K$, then through the
driver, there is a kind of cancellation of the positive force. So in general
case, we can consider this reflected BSDE as an equation with resistance,
given by \ the dependence of the driver on $K$. Since $Y_{t}\geq S_{t}$ has
to be satisfied, and $Y$ is still square integrable, the extra force from
the driver must be controlled is some sense, which is characterized the
magnitude of Lipschitz constant in $K$.

The paper is organized as following. We first recall in Section 2 the
Tanaka's formula and Skorohod's equation to give various formulate for the
increasing process $K$. In section 3, we introduce a type of reflected BSDEs
with resistance and prove the existence and uniqueness of the solution. In
section 4, the uniqueness and some properties of the solution we have
constructed are studied.

\section{Local and reflected local times}

Let $B=(B^{1},B^{2},\cdots ,B^{d})$ be a Brownian motion of dimension $d$ on
a complete probability space $(\Omega ,\mathcal{F},\mathbb{P})$, and $(%
\mathcal{F}_{t})_{t\geq 0}$ be the Brownian motion filtration associated
with $B$. Let $T>0$ be a terminal time. \ Denote by $\mathcal{P}$ the $%
\sigma $-algebra of predictable sets on $[0,T]\times \Omega $ with respect
to the filtration $(\mathcal{F}_{t})_{t\geq 0}$.

For simplicity, we introduce the following spaces of random processes over $%
(\Omega ,\mathcal{F},\mathcal{F}_{t},\mathbb{P})$. $\mathcal{L}^{2}(\mathcal{%
F}_{t})$ denotes the space of all $\mathcal{F}_{t}$-measurable, real random
variable such that $\mathbb{E}(|\eta |^{2})<\infty $, $\mathcal{M}^{2}$
denotes the space of \ (continuous) square-integrable martingales (up to
time $T$), and $\mathcal{H}_{d}^{2}(0,T)$ is the space of $\mathbb{R}^{d}$%
-valued predictable process $\psi $ such that $\mathbb{E}\int_{0}^{T}\left%
\vert \psi (t)\right\vert ^{2}dt<\infty \}$. $\mathcal{S}^{2}(0,T)$ is the
space of all continuous semimartingales (with running time $[0,T]$) over $%
(\Omega ,\mathcal{F},\mathcal{F}_{t},\mathbb{P})$, and $\mathcal{A}^{2}(0,T)$
the space of all $\mathcal{F}_{T}$-measurable, continuous and increasing
processes with initial zero such that $\mathbb{E}(K_{T}^{2})<\infty $.

If $K\in \mathcal{A}^{2}(0,T)$, then the optional projection $K^{\flat }$
(which is a right continuous modification of $t\rightarrow \mathbb{E}(K_{t}|%
\mathcal{F}_{t})$) and the dual optional projection $K^{o}$ of $K$ exist.
The dual optional projection $K^{o}$ is continuous and increasing with
initial zero, while the optional projection $K^{\flat }$ is right continuous
but not necessary increasing. Their difference $N=K^{\flat }-K^{o}$ is a
martingale which must be continuous. Hence the optional projection $K^{\flat
}$ is continuous as well. Moreover $K$ to $K^{\flat }$ is a contraction in $%
L^{p}$-norm.

The reflected backward stochastic differential equation (RBSDE or reflected
BSDE in short) considered in El Karoui, et al. \cite{MR1434123} is a
stochastic integral equation 
\begin{equation}
\text{ \ \ \ }Y_{t}=\xi
+\int_{t}^{T}f(s,Y_{s},Z_{s})ds+K_{T}-K_{t}-\int_{t}^{T}Z_{s}dB_{s}\text{ \ }
\label{rbsde-1}
\end{equation}%
for \ $t\leq T$, subject to the constrain that 
\begin{equation}
Y_{t}\geq S_{t}\text{ \ and \ }\int_{0}^{T}(Y_{t}-S_{t})dK_{t}=0\text{, \ \
\ \ \ \ \ }  \label{cons-01}
\end{equation}%
where $S$ is a continuous semimartingale such that $\sup_{t\leq T}S_{t}^{+}$
is square integrable, and $\xi \in \mathcal{L}^{2}(\mathcal{F}_{T})$, which
are given data. $f$ is called the (non-linear) driver of the reflected BSDE (%
\ref{rbsde-1}), which is global Lipschitz in $(Y,Z)$ uniformly in $t\in
\lbrack 0,T]$ and $\omega \in \Omega $.

By a solution $(Y,Z,K)$ of the terminal problem (\ref{rbsde-1})-\ref{cons-01}%
) we mean that $Y\in \mathcal{S}^{2}(0,T)$, $K\in \mathcal{A}^{2}(0,T)$ and $%
K$ is optional, and $Z\in \mathcal{H}_{d}^{2}(0,T)$, which satisfies the
stochastic integral equations (\ref{rbsde-1}) with time $t$ running from $0$
to $T$.

The constrain (\ref{cons-01}) implies that $\xi -S_{T}$ must be
non-negative, and the second condition in (\ref{cons-01}) says that $K$ has
no charge on $\{t\in \lbrack 0,T]:Y_{t}>S_{t}\}$ and increases only on $%
\{t:Y_{t}=S_{t}\}$, which is equivalent to say that $\int_{0}^{t}1_{%
\{Y_{s}-S_{s}=0\}}dK_{s}=K_{t}$ for $0\leq t\leq T$.

Since 
\begin{equation}
Y_{0}=\xi +\int_{0}^{T}f(s,Y_{s},Z_{s})ds+K_{T}-\int_{0}^{T}Z_{s}dB_{s}
\label{y0-eq}
\end{equation}%
so that 
\begin{equation}
Y_{t}=Y_{0}-\int_{0}^{t}f(s,Y_{s},Z_{s})ds-K_{t}+\int_{0}^{t}Z_{s}dB_{s}
\label{yt-eq}
\end{equation}%
and therefore the martingale part of $Y$ is $M_{t}=\int_{0}^{t}Z_{s}dB_{s}$.

Our first task is to give two representations of the increasing process $K$
in terms of sample paths $Y$ and $Z$, by using the Tanaka formula and the
Skorohod equation respectively. These representation formulate are well
known in stochastic analysis, we however employ them to the study of
existence and uniqueness for a class of reflected backward stochastic
differential equations with resistance.

\subsection{Tanaka's formula}

We wish to interpret the increasing process $K$ as a time-reversed local
time, so that $K$ will be called the reflected local time of $Y$ at $S$. To
this end, we introduce the following notation: if $X$ is a continuous
semimartingale, then $L^{X}$ denotes the local time of the continuous
semimartingale $X-S$ at zero. That is, $L^{X}$ is defined by the Tanaka
formula 
\begin{equation}
|X_{t}-S_{t}|=|X_{0}-S_{0}|+\int_{0}^{t}\text{sgn}%
(X_{s}-S_{s})d(X_{s}-S_{s})+2L_{t}^{X}  \label{x-loc}
\end{equation}%
where sgn$(r)=-1$ for $r\leq 0$ and sgn$(r)=1$ for $r>0$. According Tanaka's
formula%
\begin{equation}
(X_{t}-S_{t})^{-}=(X_{0}-S_{0})^{-}-\int_{0}^{t}1_{\{X_{s}\leq
S_{s}\}}d(X_{s}-S_{s})+L_{t}^{X}\text{.}  \label{x-loc2}
\end{equation}

\begin{proposition}
\label{prop-01}Suppose that $Y_{t}=\int_{0}^{t}Z_{s}dB_{s}+V_{t}$ and $%
S_{t}=\int_{0}^{t}\sigma _{s}dB_{s}+A_{t}$ are two continuous
semimartingales, where $V$ and $A$ are continuous, adapted with finite
variations, and suppose that $Y\geq S$, then%
\begin{equation}
L_{t}^{Y}=\int_{0}^{t}1_{\{Y_{s}=S_{s}\}}d(V_{s}-A_{s})  \label{j-re-02}
\end{equation}%
and%
\begin{equation}
1_{\{Y_{t}=S_{t}\}}(Z_{t}-\sigma _{t})=0\text{.}  \label{j-re-03}
\end{equation}
\end{proposition}

\begin{proof}
Since $Y-S\geq 0$, so that $(Y-S)^{-}=0$. According to the Tanaka formula%
\begin{equation*}
(Y_{t}-S_{t})^{-}=(Y_{0}-S_{0})^{-}-\int_{0}^{t}1_{\{Y_{s}\leq
S_{s}\}}d(Y_{s}-S_{s})+L_{t}^{Y}
\end{equation*}%
so that%
\begin{eqnarray*}
L_{t}^{Y} &=&\int_{0}^{t}1_{\{Y_{s}=S_{s}\}}d(Y_{s}-S_{s}) \\
&=&\int_{0}^{t}1_{\{Y_{s}=S_{s}\}}d(V_{s}-A_{s})+\int_{0}^{t}1_{%
\{Y_{s}=S_{s}\}}(Z_{s}-\sigma _{s})dB_{s}\text{.}
\end{eqnarray*}%
The martingale part must be zero as $L$ is increasing, which yields (\ref%
{j-re-03}), and therefore (\ref{j-re-02}) follows as well.
\end{proof}

The following lemma demonstrates in some sense $K$ is the time inverse of
the local time of $Y-S$.

\begin{corollary}
\label{coro-001}Assume that $Y\geq S$ are two continuous semimartingale, 
\begin{equation}
\text{ \ \ \ }Y_{t}=Y_{0}-\int_{0}^{t}f_{s}ds-K_{t}+\int_{0}^{t}Z_{s}dB_{s}
\label{we-05}
\end{equation}
and $S=N+A$ ($N$ is the martingale part of $S$ and $A$ is its variation
part), where $(f_{t})_{t\in \lbrack 0,T]}$ is optional and $\mathbb{E}%
\int_{0}^{T}f_{s}^{2}ds<\infty $, $Z\in \mathcal{H}^{d}(0,T)$, $Y_{0}\in 
\mathcal{L}^{2}(\mathcal{F}_{0})$, $K\in \mathcal{A}^{2}(0,T)$ is adapted,
such that $\int_{0}^{t}1_{\{Y_{s}=S_{s}\}}dK_{s}=K_{t}$. Then%
\begin{equation}
K_{t}=-\int_{0}^{t}1_{\{Y_{s}=S_{s}\}}f_{s}ds-\int_{0}^{t}1_{\{Y_{s}=S_{s}%
\}}dA_{s}-L_{t}^{Y}  \label{we-06}
\end{equation}%
and%
\begin{eqnarray}
K_{t}
&=&-\int_{0}^{t}1_{\{Y_{s}=S_{s}\}}f_{s}ds-\int_{0}^{t}1_{\{Y_{s}=S_{s}%
\}}dY_{s}  \notag \\
&&+\int_{0}^{t}1_{\{Y_{s}=S_{s}\}}dN_{s}\text{,}  \label{re-p1}
\end{eqnarray}%
where $N$ is the martingale part of $S$.
\end{corollary}

\subsection{Skorohod's equation}

The most useful form for $K$ to our study is however the representation
formula given by the Skorohod equation.

Again we assume that $Y\geq S$ are two continuous semimartingales, and $Y$
is given by (\ref{we-05}). Let $y_{t}=Y_{T-t}-S_{T-t}$, $L_{t}=K_{T}-K_{T-t}$
and 
\begin{equation}
x_{t}=\int_{T-t}^{T}f_{s}ds-\int_{T-t}^{T}Z_{s}dB_{s}+S_{T}-S_{T-t}\text{.}
\label{xt-eq}
\end{equation}%
Then $L_{0}=0$, $t\rightarrow L_{t}$ increases only on $\{t:y_{t}=0\}$, $%
y_{t}\geq 0$, $\eta =Y_{T}-S_{T}\geq 0$, $x_{0}=0$, and 
\begin{equation}
y_{t}=\eta +x_{t}+L_{t}\text{ .}  \label{yt-eq1}
\end{equation}

According to Skorohod's equation (Lemma 6.14, page 210 in \cite{MR1121940},
with the convention that $x_{t}=x_{T}$, $y_{t}=y_{T}$ and $L_{t}=L_{T}$ for $%
t\geq T$) 
\begin{equation}
L_{t}=\max \left[ 0,\max_{0\leq s\leq t}\left\{ -\left( \eta +x_{s}\right)
\right\} \right] \text{, \ }\forall t\geq 0\text{.}  \label{skorohod1}
\end{equation}%
That is 
\begin{equation}
L_{t}=\max \left[ 0,\max_{T-t\leq s\leq T}\left\{ -\left(
Y_{T}+\int_{s}^{T}f_{r}dr-S_{s}-\int_{s}^{T}Z_{r}dB_{r}\right) \right\} %
\right]  \label{local-r1}
\end{equation}%
for $0\leq t\leq T$, and we may recover $K_{t}=L_{T}-L_{T-t}$ to obtain 
\begin{eqnarray}
K_{t} &=&\max \left[ 0,\max_{0\leq s\leq T}\left\{ -\left(
Y_{T}+\int_{s}^{T}f_{r}-S_{s}-\int_{s}^{T}Z_{r}dB_{r}\right) \right\} \right]
\notag \\
&&-\max \left[ 0,\max_{t\leq s\leq T}\left\{ -\left(
Y_{T}+\int_{s}^{T}f_{r}dr-S_{s}-\int_{s}^{T}Z_{r}dB_{r}\right) \right\} %
\right] \text{.}  \label{local-r2}
\end{eqnarray}

\section{Reflected BSDE with resistance}

The representation formula (\ref{local-r2}) for the reflected local time $K$
may be used to study a class of reflected backward stochastic differential
equations with non-linear resistance caused by the reflected local time $K$.
In this paper we study the following stochastic integral equation 
\begin{equation}
\text{ \ \ \ }Y_{t}=\xi
+\int_{t}^{T}f(s,Y_{s},Z_{s},K_{s})ds+K_{T}-K_{t}-\int_{t}^{T}Z_{s}dB_{s}%
\text{ \ }  \label{rbsde-ge}
\end{equation}%
for \ $t\leq T$, subject to the constrain that 
\begin{equation}
Y_{t}\geq S_{t}\text{ \ and \ }\int_{0}^{T}(Y_{t}-S_{t})dK_{t}=0\text{,\ \ \
\ }  \label{cons-02}
\end{equation}%
$S$ is a continuous semimartingale such that $\sup_{t\leq T}S_{t}^{+}$ is
square integrable, and $\xi \in \mathcal{L}^{2}(\mathcal{F}_{T})$, which are
given data.

Assume that $f$ is global Lipschitz continuous 
\begin{equation}
|f(s,y,z,k)-f(s,y^{\prime },z^{\prime },k^{\prime })|\leq C_{1}(|y-y^{\prime
}|+|z-z^{\prime }|)+C_{2}|k-k^{\prime }|  \label{lip-01}
\end{equation}%
for all $y,y^{\prime },z,z^{\prime },k,k^{\prime }$, where $C_{1}$ and $%
C_{2} $ are two constants, and $\mathbb{E}\int_{0}^{T}f^{0}(t)^{2}dt<\infty $%
, where%
\begin{equation}
f^{0}(t)\equiv f(t,0,0,0)\text{.}  \label{f-0}
\end{equation}

By a solution triple $(Y,Z,K)$ of the terminal problem (\ref{rbsde-ge}) we
mean that $Y\in \mathcal{S}^{2}(0,T)$, $K\in \mathcal{A}^{2}(0,T)$ and $K$
is optional, and $Z\in \mathcal{H}_{d}^{2}(0,T)$, which satisfies the
stochastic integral equations (\ref{rbsde-ge}) with time $t$ running from $0$
to $T$.

An additional feature over the reflected BSDE\ (\ref{rbsde-1}) is the
dependence of the driver with respect to the reflected local time $K$. The
integral equation (\ref{rbsde-ge}) is not local in time, since $K$ will be
path dependent over the whole range $[0,T]$. This is the reason why we have
to require the Lipschitz constant $C_{2}$ in (\ref{lip-01}) to be small.

According to (\ref{local-r2}), if $(Y,Z,K)$ is a solution of the problem (%
\ref{rbsde-ge})-(\ref{cons-02}), then we must have 
\begin{eqnarray}
K_{t} &=&\max \left[ 0,\max_{0\leq s\leq T}\left\{ -\left( \xi
-S_{s}+\int_{s}^{T}f(r,Y_{r},Z_{r},K_{r})dr-\int_{s}^{T}Z_{r}dB_{r}\right)
\right\} \right]  \notag \\
&&-\max \left[ 0,\max_{t\leq s\leq T}\left\{ -\left( \xi
-S_{s}+\int_{s}^{T}f(r,Y_{r},Z_{r},K_{r})dr-\int_{s}^{T}Z_{r}dB_{r}\right)
\right\} \right]  \label{k-local-l1}
\end{eqnarray}%
which is already noticed in \cite{MR1434123}. What is new here is the use of
this formula to build a proper Picard iteration associated with (\ref%
{rbsde-ge}), which is described in the following section.

\begin{remark}
For the reflected BSDE with resistance (\ref{rbsde-ge}), we can still write
the solution $Y$ as the value process of\ an optimal stopping problem
following the same arguments of Proposition 2.3 in \cite{MR1434123},
noticing that $K$ is continuous in $t$. However, here it is not a standard
optimal stopping problem, and we can not get the solution given by Snell
envelope, which is the essential step of Picard's iteration used in \cite%
{MR1434123}.
\end{remark}

\subsection{Constructing Picard's iteration}

We show the existence of a unique solution by constructing an appropriate
(non)-linear mapping defined by the stochastic integral equation (\ref%
{rbsde-ge}), so that the unique solution is given as its fixed point.

Let us develop the iteration procedure as following. Suppose $Y\in \mathcal{S%
}^{2}(0,T)$, $Z\in \mathcal{H}_{d}^{2}(0,T)$ and $K\in \mathcal{A}^{2}(0,T)$%
, $Y\geq S$. After iteration once we will obtain 
\begin{equation*}
(\tilde{Y},\tilde{Z},\tilde{K})\in \mathcal{S}^{2}(0,T)\times \mathcal{H}%
_{d}^{2}(0,T)\times \mathcal{A}^{2}(0,T)\text{,}
\end{equation*}%
and $\tilde{Z}.B$ is the martingale part of $\tilde{Y}$. Thus we can assume
from the beginning, without losing generality, that $M_{t}-M_{0}=%
\int_{0}^{t}Z_{s}dB_{s}$ is the martingale part of $Y$, although we prefer
consider three processes $Y,Z,K$ as independent variables. According to (\ref%
{k-local-l1}) we first define 
\begin{eqnarray}
\tilde{K}_{t} &=&\max \left[ 0,\max_{0\leq s\leq T}\left\{ -\left( \xi
+\int_{s}^{T}f(r,Y_{r},Z_{r},K_{r}^{\flat
})dr-S_{s}-\int_{s}^{T}Z_{r}dB_{r}\right) \right\} \right]  \notag \\
&&-\max \left[ 0,\max_{t\leq s\leq T}\left\{ -\left( \xi
+\int_{s}^{T}f(r,Y_{r},Z_{r},K_{r}^{\flat
})dr-S_{s}-\int_{s}^{T}Z_{r}dB_{r}\right) \right\} \right]  \label{ktilde-01}
\end{eqnarray}%
where $K$ is replaced by the optional projection of $K$, as we do not assume 
$K$ is optional, but we want to ensure the arguments in the driver $f$ are
optional.

We are going to define $\tilde{M}$ and $\tilde{Y}$. The natural way to
define $\tilde{Y}$ is to take the right-hand side of (\ref{rbsde-ge}) which
is 
\begin{equation}
\hat{Y}_{t}=\xi +\int_{t}^{T}f(s,Y_{s},Z_{s},K_{s}^{\flat })ds+\tilde{K}_{T}-%
\tilde{K}_{t}-\int_{t}^{T}Z_{s}dB_{s}\text{.}  \label{yhat-eq1}
\end{equation}%
$\hat{Y}$ is however not not necessary adapted. Therefore we define $\tilde{Y%
}$ to be its optional projection $\hat{Y}^{\flat }$. That is, 
\begin{eqnarray}
\tilde{Y}_{t} &=&\mathbb{E}\left\{ \left. \xi
+\int_{t}^{T}f(s,Y_{s},Z_{s},K_{s}^{\flat })ds+\tilde{K}_{T}-\tilde{K}%
_{t}-\int_{t}^{T}Z_{s}dB_{s}\right\vert \mathcal{F}_{t}\right\}  \notag \\
&=&\mathbb{E}\left\{ \left. \xi +\int_{t}^{T}f(s,Y_{s},Z_{s},K_{s}^{\flat
})ds+\tilde{K}_{T}-\tilde{K}_{t}\right\vert \mathcal{F}_{t}\right\} \text{.}
\label{ytilde-eq1}
\end{eqnarray}%
According to Skorohod's equation, $\hat{Y}\geq S$, so is $\tilde{Y}$.
Therefore the mapping $Y\rightarrow \tilde{Y}$ preserves the constraint $%
\tilde{Y}\geq S$. Moreover $\tilde{K}$ increases only on $\{t:\hat{Y}%
_{t}-S_{t}=0\}$, which does not necessarily coincide with level set $\{t:%
\tilde{Y}_{t}-S_{t}=0\}$.

On the other hand, according to (\ref{ytilde-eq1}), the\ semimartingale
decomposition of $\tilde{Y}$ is given by 
\begin{eqnarray}
\tilde{Y}_{t} &=&\mathbb{E}\left\{ \left. \xi +\tilde{K}_{T}+%
\int_{0}^{T}f(s,Y_{s},Z_{s},K_{s}^{\flat })ds\right\vert \mathcal{F}%
_{t}\right\} -\tilde{N}_{t}  \notag \\
&&-\tilde{K}_{t}^{o}-\int_{0}^{t}f(s,Y_{s},Z_{s},K_{s}^{\flat })ds
\label{y-tilded-de-01}
\end{eqnarray}%
where $\tilde{N}_{t}=\tilde{K}_{t}^{\flat }-\tilde{K}_{t}^{o}$ is a
continuous martingale. Therefore the martingale part of $\tilde{Y}$ is given
by 
\begin{equation}
\tilde{M}_{t}=\mathbb{E}\left\{ \left. \xi +\tilde{K}_{T}+%
\int_{0}^{T}f(s,Y_{s},Z_{s},K_{s}^{\flat })ds\right\vert \mathcal{F}%
_{t}\right\} -\tilde{N}_{t}  \label{mtilde-01}
\end{equation}%
which in turn defines the density predictable process $\tilde{Z}$ by It\^{o}%
's martingale representation $\tilde{M}_{t}-\tilde{M}_{0}=\int_{0}^{t}\tilde{%
Z}_{s}.dB_{s}$, so that%
\begin{equation}
\tilde{Y}_{t}=\xi +\int_{t}^{T}f(s,Y_{s},Z_{s},K_{s}^{\flat })ds+\tilde{K}%
_{T}^{o}-\tilde{K}_{t}^{o}-\int_{t}^{T}\tilde{Z}_{s}.dB_{s}\text{.}
\label{t-y-03}
\end{equation}%
It is clear that from the definition and the Lipschitz condition (\ref%
{lip-01}) 
\begin{equation*}
(\tilde{Y},\tilde{Z},\tilde{K})\in \mathcal{S}^{2}(0,T)\times \mathcal{H}%
_{d}^{2}(0,T)\times \mathcal{A}^{2}(0,T)\text{.}
\end{equation*}%
The mapping $\mathfrak{L}:(Y,Z,K)\rightarrow (\tilde{Y},\tilde{Z},\tilde{K})$
is thus well defined.

It seems reasonable to have the optional \emph{dual} projection $K^{o}$ in
place of the optional projection $K^{\flat }$ in defining $\tilde{Y}$ by (%
\ref{y-tilded-de-01}). The reason we prefer the optional projection lies in
the fact that $X\rightarrow X^{\flat }$ is a contraction in $L^{p}$-space,
but $X\rightarrow X^{o}$ is not.

\begin{proposition}
\label{prop-03}If $(Y,Z,K)$ is a fixed point of $\mathfrak{L}$, then $%
(Y,Z,K) $ is a solution the reflected BSDE (\ref{rbsde-ge})-(\ref{cons-02}).
\end{proposition}

\begin{proof}
Suppose $(Y,Z,K)$ is a fixed point of the non-linear mapping $\mathfrak{L}$,
so that 
\begin{equation*}
M_{t}=\mathbb{E}\left\{ \xi +\int_{0}^{T}f(s,Y_{s},Z_{s},K_{s}^{\flat
})ds+K_{T}-K_{t}|\mathcal{F}_{t}\right\} +K_{t}^{o}\text{,}
\end{equation*}%
and 
\begin{equation*}
Y_{t}=\mathbb{E}\left\{ \xi +\int_{t}^{T}f(s,Y_{s},Z_{s},K_{s}^{\flat
})ds+K_{T}-K_{t}|\mathcal{F}_{t}\right\}
\end{equation*}%
where 
\begin{eqnarray*}
K_{t} &=&\max \left[ 0,\max_{0\leq s\leq T}\left\{ -\left( \xi
+\int_{s}^{T}f(r,Y_{r},Z_{r},K_{r}^{\flat
})dr-S_{s}-\int_{s}^{T}Z_{r}dB_{r}\right) \right\} \right] \\
&&-\max \left[ 0,\max_{t\leq s\leq T}\left\{ -\left( \xi
+\int_{s}^{T}f(r,Y_{r},Z_{r},K_{r}^{\flat
})dr-S_{s}-\int_{s}^{T}Z_{r}dB_{r}\right) \right\} \right] \text{.}
\end{eqnarray*}%
Then $Y_{T}=\xi $ and 
\begin{equation*}
Y_{t}=M_{t}-K_{t}^{o}-\int_{0}^{t}f(s,Y_{s},Z_{s},K_{r}^{\flat })ds\text{,}
\end{equation*}%
so that 
\begin{equation*}
\xi -Y_{t}=\int_{t}^{T}Z_{s}dB_{s}-\left( K_{T}^{o}-K_{t}^{o}\right)
-\int_{t}^{T}f(s,Y_{s},Z_{s},K_{r}^{\flat })ds\text{ .}
\end{equation*}%
According to the uniqueness of the Skorohod's equation, it follows that $%
K^{o}=K$. therefore $K$ is adapted, and $K=K^{\flat }=K^{o}$. That completes
the proof.
\end{proof}

\subsection{Main estimates}

Let us develop several a priori estimates for $\mathfrak{L}(Y,Z,K)=(\tilde{Y}%
,\tilde{Z},\tilde{K})$. We begin with an elementary fact:

\begin{lemma}
\label{lem-c1}Let $\varphi ,\psi $ be two continuous paths in $\mathbb{R}%
^{1} $. Then 
\begin{equation*}
\left| \sup_{s\leq t}\varphi _{s}-\sup_{s\leq t}\psi _{s}\right| \leq
\sup_{s\leq t}|\varphi _{s}-\psi _{s}|\text{.}
\end{equation*}
\end{lemma}

\begin{proof}
$\delta =\sup_{s\leq t}|\varphi _{s}-\psi _{s}|$. Then 
\begin{equation*}
\varphi _{s}\leq \psi _{s}+\delta \leq \sup_{s\leq t}\psi _{s}+\delta
\end{equation*}
for any $s\leq t$, so that $\sup_{s\leq t}\varphi _{s}\leq \sup_{s\leq
t}\psi _{s}+\delta $. Similarly $\sup_{s\leq t}\psi _{s}\leq \sup_{s\leq
t}\varphi _{s}+\delta $.
\end{proof}

Suppose $(Y,Z,K)$, $(Y^{\prime },Z^{\prime },K^{\prime })\in \mathcal{S}%
^{2}\times \mathcal{H}_{d}^{2}\times \mathcal{A}^{2}$ such that $%
Y_{T}=Y_{T}^{\prime }=\xi $, and $Y\geq S$, $Y^{\prime }\geq S$.

Let us prove the following key a priori estimate about $\mathfrak{L}$. Let $(%
\tilde{Y},\tilde{Z},\tilde{K})=\mathfrak{L}(Y,Z,K)$ and $(\tilde{Y}^{\prime
},\tilde{Z}^{\prime },\tilde{K}^{\prime })=\mathfrak{L}(Y^{\prime
},Z^{\prime },K^{\prime })$. Let $\alpha \geq 0$ to be chosen late, and let $%
D_{t}=e^{\alpha t}|Y_{t}-Y_{t}^{\prime }|^{2}$ and $\tilde{D}_{t}=e^{\alpha
t}|\tilde{Y}_{t}-\tilde{Y}_{t}^{\prime }|^{2}$.

\begin{proposition}
Suppose $f$ satisfies the Lipschitz condition (\ref{lip-01}). Then for any $%
\alpha \geq 0$, $\varepsilon >0$ and $\varepsilon ^{\prime }>0$ we have%
\begin{eqnarray}
\mathbb{E}\left( \tilde{D}_{0}\right) &\leq &-(\alpha -\varepsilon
C_{1}-\varepsilon ^{\prime }C_{2})||\tilde{Y}-\tilde{Y}^{\prime }||_{\alpha
}^{2}-||\tilde{Z}-\tilde{Z}^{\prime }||_{\alpha }^{2}  \notag \\
&&+\frac{2C_{1}}{\varepsilon }\left( ||Y-Y^{\prime }||_{\alpha
}^{2}+||Z-Z^{\prime }||_{\alpha }^{2}\right)  \notag \\
&&+\frac{2C_{2}}{\varepsilon ^{\prime }}||K^{\flat }-K^{\prime }{}^{\flat
}||_{\alpha }^{2}  \label{key-est-a2}
\end{eqnarray}%
where $C_{1},C_{2}$ are the Lipschitz constants appearing in (\ref{lip-01}).
\end{proposition}

\begin{proof}
According to (\ref{y-tilded-de-01}) 
\begin{eqnarray}
\tilde{Y}_{t}-\tilde{Y}_{t}^{\prime } &=&\left( \tilde{M}_{t}-\tilde{M}%
_{t}^{\prime }\right) -(\tilde{K}_{t}^{o}-\tilde{K}_{t}^{\prime o})  \notag
\\
&&-\int_{0}^{t}\left( f(s,Y_{s},Z_{s},K_{s}^{\flat })-f(s,Y_{s}^{\prime
},Z_{s}^{\prime },K^{\prime }{}_{s}^{\flat })\right) ds  \label{D-teq1}
\end{eqnarray}%
where $\tilde{M}$ (resp. $\tilde{M}^{\prime }$) is the martingale part of $%
\tilde{Y}$ (resp. $\tilde{Y}^{\prime }$), given by (\ref{mtilde-01}), so
that by It\^{o}'s formula, 
\begin{eqnarray}
\tilde{D}_{t} &=&-\int_{t}^{T}e^{\alpha s}d\left( \tilde{Y}_{s}-\tilde{Y}%
_{s}^{\prime }\right) ^{2}-\alpha \int_{t}^{T}e^{\alpha s}\left( \tilde{Y}%
_{s}-\tilde{Y}_{s}^{\prime }\right) ^{2}ds  \notag \\
&=&-\int_{t}^{T}e^{\alpha s}d\langle \tilde{M}-\tilde{M}^{\prime }\rangle
_{s}-\alpha \int_{t}^{T}e^{\alpha s}\left( \tilde{Y}_{s}-\tilde{Y}%
_{s}^{\prime }\right) ^{2}ds  \notag \\
&&-\int_{t}^{T}2e^{\alpha s}\left( \tilde{Y}_{s}-\tilde{Y}_{s}^{\prime
}\right) d\left( \tilde{Y}_{s}-\tilde{Y}_{s}^{\prime }\right)  \notag \\
&=&-\alpha \int_{t}^{T}\tilde{D}_{s}ds-\int_{t}^{T}e^{\alpha s}d\langle 
\tilde{M}-\tilde{M}^{\prime }\rangle _{s}-2\int_{t}^{T}e^{\alpha s}\left( 
\tilde{Y}_{s}-\tilde{Y}_{s}^{\prime }\right) d\left( \tilde{M}_{s}-\tilde{M}%
_{s}^{\prime }\right)  \notag \\
&&+2\int_{t}^{T}e^{\alpha s}\left( \tilde{Y}_{s}-\tilde{Y}_{s}^{\prime
}\right) d\left( \tilde{K}_{s}^{o}-\tilde{K}_{s}^{\prime o}\right)  \notag \\
&&+2\int_{t}^{T}e^{\alpha s}\left( \tilde{Y}_{s}-\tilde{Y}_{s}^{\prime
}\right) \left( f(s,Y_{s},Z_{s},K_{s}^{\flat })-f(s,Y_{s}^{\prime
},Z_{s}^{\prime },K^{\prime }{}_{s}^{\flat })\right) ds\text{.}
\label{kt-01}
\end{eqnarray}%
Taking expectation to obtain 
\begin{eqnarray}
\mathbb{E}\tilde{D}_{t} &=&-\alpha \int_{t}^{T}\mathbb{E}\left( \tilde{D}%
_{s}\right) ds-\mathbb{E}\int_{t}^{T}e^{\alpha s}d\langle \tilde{M}-\tilde{M}%
^{\prime }\rangle _{s}  \notag \\
&&+2\mathbb{E}\int_{t}^{T}e^{\alpha s}\left( \tilde{Y}_{s}-\tilde{Y}%
_{s}^{\prime }\right) d(\tilde{K}_{s}-\tilde{K}_{s}^{\prime })  \notag \\
&&+2\int_{t}^{T}\mathbb{E}\left\{ e^{\alpha s}\left( \tilde{Y}_{s}-\tilde{Y}%
_{s}^{\prime }\right) \left[ f(s,Y_{s},Z_{s},K_{s}^{\flat
})-f(s,Y_{s}^{\prime },Z_{s}^{\prime },K^{\prime }{}_{s}^{\flat })\right]
\right\} ds\text{,}  \label{hg-tr1}
\end{eqnarray}%
where we have used the fact that 
\begin{equation*}
\mathbb{E}\int_{t}^{T}\varphi _{s}d\left( \tilde{K}_{s}^{o}-\tilde{K}%
_{s}^{\prime o}\right) =\mathbb{E}\int_{t}^{T}\varphi _{s}d\left( \tilde{K}%
_{s}-\tilde{K}_{s}^{\prime }\right)
\end{equation*}%
for an optional process $\varphi $. Now we use an important observation due
to \cite{MR1434123}, that is, 
\begin{eqnarray*}
&&\mathbb{E}\int_{t}^{T}e^{\alpha s}\left( \tilde{Y}_{s}-\tilde{Y}%
_{s}^{\prime }\right) d(\tilde{K}_{s}-\tilde{K}_{s}^{\prime }) \\
&=&\mathbb{E}\int_{t}^{T}e^{\alpha s}(\tilde{Y}_{s}-S_{s})d\tilde{K}_{s}+%
\mathbb{E}\int_{t}^{T}e^{\alpha s}(\tilde{Y}_{s}^{\prime }-S_{s})d\tilde{K}%
_{s}^{\prime } \\
&&-\mathbb{E}\int_{t}^{T}e^{\alpha s}(\tilde{Y}_{s}-S_{s})d\tilde{K}%
_{s}^{\prime }-\mathbb{E}\int_{t}^{T}e^{\alpha s}(\tilde{Y}_{s}^{\prime
}-S_{s})d\tilde{K}_{s} \\
&\leq &\mathbb{E}\int_{t}^{T}e^{\alpha s}(\tilde{Y}_{s}-S_{s})d\tilde{K}_{s}+%
\mathbb{E}\int_{t}^{T}e^{\alpha s}(\tilde{Y}_{s}^{\prime }-S_{s})d\tilde{K}%
_{s}^{\prime }\text{.}
\end{eqnarray*}%
Moreover, according to Skorohod's equation, $\tilde{K}$ increases only on $%
\{s:\hat{Y}_{s}-S_{s}=0\}$ so that 
\begin{equation*}
\mathbb{E}\int_{t}^{T}e^{\alpha s}(\hat{Y}_{s}-S_{s})d\tilde{K}_{s}=0\text{.}
\end{equation*}%
Since $\tilde{Y}$ is the optional projection, and $\tilde{K}^{o}$ is the
dual optional projection of $\tilde{K}$, therefore 
\begin{eqnarray*}
\mathbb{E}\int_{t}^{T}e^{\alpha s}(\tilde{Y}_{s}-S_{s})d\tilde{K}_{s} &=&%
\mathbb{E}\int_{t}^{T}e^{\alpha s}(\tilde{Y}_{s}-S_{s})d\tilde{K}_{s}^{o} \\
&=&\mathbb{E}\left( \int_{t}^{T}e^{\alpha s}(\hat{Y}_{s}-S_{s})d\tilde{K}%
_{s}\right) ^{o}\text{.}
\end{eqnarray*}%
Since $\tilde{K}$ increases only on $\{s:\hat{Y}_{s}-S_{s}=0\}$, so that $%
\int_{t}^{T}e^{\alpha s}(\hat{Y}_{s}-S_{s})d\tilde{K}_{s}=0$ and therefore $%
\mathbb{E}\int_{t}^{T}e^{\alpha s}(\tilde{Y}_{s}-S_{s})d\tilde{K}_{s}=0$.
Similarly $\mathbb{E}\int_{t}^{T}e^{\alpha s}(\tilde{Y}_{s}^{\prime }-S_{s})d%
\tilde{K}_{s}^{\prime }=0$. Hence 
\begin{equation*}
\mathbb{E}\int_{t}^{T}e^{\alpha s}\left( \tilde{Y}_{s}-\tilde{Y}_{s}^{\prime
}\right) d(\tilde{K}_{s}-\tilde{K}_{s}^{\prime })\leq 0\text{ .}
\end{equation*}%
Putting this estimate into (\ref{hg-tr1}) to obtain 
\begin{eqnarray}
\mathbb{E}\left( \tilde{D}_{t}\right) &\leq &-\alpha \int_{t}^{T}\mathbb{E}%
\left( \tilde{D}_{s}\right) ds-\mathbb{E}\int_{t}^{T}e^{\alpha s}d\langle 
\tilde{M}-\tilde{M}^{\prime }\rangle _{s}  \notag \\
&&+2\int_{t}^{T}e^{\alpha s}\mathbb{E}\left\{ \left( \tilde{Y}_{s}-\tilde{Y}%
_{s}^{\prime }\right) \left[ f(s,Y_{s},Z_{s},K_{s}^{\flat
})-f(s,Y_{s}^{\prime },Z_{s}^{\prime },K^{\prime }{}_{s}^{\flat })\right]
\right\} ds\text{.}  \label{gh-j1}
\end{eqnarray}%
We use the global Lipschitz continuity of $f$ to handle the last integral on
the right-hand side of (\ref{gh-j1}), the method however is standard. Indeed%
\begin{eqnarray}
\mathbb{E}\left( \tilde{D}_{t}\right) &\leq &-\alpha \int_{t}^{T}\mathbb{E}%
\left( \tilde{D}_{s}\right) ds-\mathbb{E}\int_{t}^{T}e^{\alpha s}|\tilde{Z}%
_{s}-\tilde{Z}_{s}|^{2}ds  \notag \\
&&+2C_{1}\int_{t}^{T}e^{\alpha s}\mathbb{E}\left( |\tilde{Y}_{s}-\tilde{Y}%
_{s}^{\prime }|\left( |Y_{s}-Y_{s}^{\prime }|+|Z_{s}-Z_{s}^{\prime }|\right)
\right) ds  \notag \\
&&+2C_{2}\int_{t}^{T}e^{\alpha s}\mathbb{E}\left( |\tilde{Y}_{s}-\tilde{Y}%
_{s}^{\prime }||K_{s}^{\flat }-K^{\prime }{}_{s}^{\flat }|\right) ds  \notag
\\
&\leq &-(\alpha -\varepsilon C_{1}-\varepsilon ^{\prime }C_{2})\int_{t}^{T}%
\mathbb{E}\left( \tilde{D}_{s}\right) ds-\mathbb{E}\int_{t}^{T}e^{\alpha s}|%
\tilde{Z}_{s}-\tilde{Z}_{s}|^{2}ds  \notag \\
&&+\frac{2C_{1}}{\varepsilon }\mathbb{E}\int_{t}^{T}e^{\alpha s}\left(
|Y_{s}-Y_{s}^{\prime }|^{2}+|Z_{s}-Z_{s}^{\prime }|^{2}\right) ds  \notag \\
&&+\frac{2C_{2}}{\varepsilon ^{\prime }}\mathbb{E}\int_{t}^{T}e^{\alpha
s}|K_{s}^{\flat }-K^{\prime }{}_{s}^{\flat }|^{2}ds\text{.}  \label{uy-01}
\end{eqnarray}%
Choosing $t=0$ we deduce the required estimate.
\end{proof}

The next estimate is also essential in this article.

\begin{proposition}
We have%
\begin{eqnarray}
||\tilde{K}-\tilde{K}^{\prime }||_{\infty }^{2} &\leq &\left(
24TC_{1}^{2}+4C_{3}\right) \left( ||Y-Y^{\prime }||_{0}^{2}+||Z-Z^{\prime
}||_{0}^{2}\right)  \notag \\
&&+24T^{2}C_{1}^{2}||K-K^{\prime }||_{\infty }^{2}  \label{k-est-b1}
\end{eqnarray}%
where $||K-K^{\prime }||_{\infty }^{2}=\sup_{0\leq t\leq T}\mathbb{E}%
|K_{s}-K_{s}^{\prime }|^{2}$, where $C_{3}$ is the constant appearing in the
Burkholder inequality.
\end{proposition}

\begin{proof}
Recall that%
\begin{eqnarray*}
\tilde{K}_{t} &=&\max \left[ 0,\max_{0\leq s\leq T}\left\{ -\left( \xi
+\int_{s}^{T}f(r,Y_{r},Z_{r},K_{r}^{\flat
})dr-S_{s}-\int_{s}^{T}Z_{r}dB_{r}\right) \right\} \right] \\
&&-\max \left[ 0,\max_{t\leq s\leq T}\left\{ -\left( \xi
+\int_{s}^{T}f(r,Y_{r},Z_{r},K_{r}^{\flat
})dr-S_{s}-\int_{s}^{T}Z_{r}dB_{r}\right) \right\} \right]
\end{eqnarray*}%
which yields that%
\begin{eqnarray*}
|\tilde{K}_{t}-\tilde{K}_{t}^{\prime }|^{2} &\leq &4T\int_{0}^{T}\left\vert
f(s,Y_{s},Z_{s},K_{s}^{\flat })-f(s,Y_{s}^{\prime },Z_{s}^{\prime
},K_{s}^{\prime \flat })\right\vert ^{2}ds \\
&&+4\left\vert \sup_{0\leq s\leq T}\int_{s}^{T}(Z_{r}-Z_{r}^{\prime
})dB_{r}\right\vert ^{2} \\
&\leq &24TC_{1}^{2}\int_{0}^{T}\left( \left\vert Y_{s}-Y_{s}^{\prime
}\right\vert ^{2}+\left\vert Z_{s}-Z_{s}^{\prime }\right\vert ^{2}\right) ds
\\
&&+24TC_{2}^{2}\int_{0}^{T}|K_{s}^{\flat }-K_{s}^{\prime \flat
}|^{2}ds+4\left\vert \sup_{0\leq s\leq T}\int_{0}^{T}(Z_{r}-Z_{r}^{\prime
})dB_{r}\right\vert ^{2}
\end{eqnarray*}%
so that 
\begin{eqnarray*}
\mathbb{E}|\tilde{K}_{t}-\tilde{K}_{t}^{\prime }|^{2} &\leq &\left(
24TC_{1}^{2}+4C_{3}\right) \mathbb{E}\int_{0}^{T}\left( \left\vert
Y_{s}-Y_{s}^{\prime }\right\vert ^{2}+\left\vert Z_{s}-Z_{s}^{\prime
}\right\vert ^{2}\right) ds \\
&&+24TC_{1}^{2}\int_{0}^{T}\mathbb{E}|K_{s}^{\flat }-K_{s}^{\prime \flat
}|^{2}ds \\
&\leq &\left( 24TC_{1}^{2}+4C_{3}\right) \mathbb{E}\int_{0}^{T}\left(
\left\vert Y_{s}-Y_{s}^{\prime }\right\vert ^{2}+\left\vert
Z_{s}-Z_{s}^{\prime }\right\vert ^{2}\right) ds \\
&&+24TC_{2}^{2}\int_{0}^{T}\mathbb{E}|K_{s}-K_{s}^{\prime }|^{2}ds
\end{eqnarray*}%
which implies (\ref{k-est-b1}).
\end{proof}

\subsection{Existence theorem}

We are now in a position to show the existence of a solution to (\ref%
{rbsde-ge})-(\ref{cons-02}).

\begin{theorem}
\label{ex-01}There is a constant $C_{0}>0$ depending only on $C_{1}$ such
that if $C_{2}T\leq C_{0}$, where $C_{2}$ is the Lipschitz constant
appearing in (\ref{lip-01}), then there is a unique solution $(Y,Z,K)$ to
the problem (\ref{rbsde-ge})-(\ref{cons-02}). Moreover the reversed local
time satisfies (\ref{k-local-l1}). If $C_{2}=0$ that is the driver $f$ does
not depend on $K$, then there is no restriction on $T$.
\end{theorem}

\begin{proof}
Let $\alpha \geq 0$ and $\beta >0$ to be chosen late, and define%
\begin{equation*}
||(Y,Z,K)-(Y^{\prime },Z^{\prime },K^{\prime })||_{\alpha ,\beta
}^{2}=||Y-Y^{\prime }||_{\alpha }^{2}+||Z-Z^{\prime }||_{\alpha }^{2}+\beta
||K-K^{\prime }||_{\infty }^{2}\text{.}
\end{equation*}%
Let $(\tilde{Y},\tilde{Z},\tilde{K})=\mathfrak{L}(Y,Z,K)$ and $(\tilde{Y}%
^{\prime },\tilde{Z}^{\prime },\tilde{K}^{\prime })=\mathfrak{L}(Y^{\prime
},Z^{\prime },K^{\prime })$. Then%
\begin{equation*}
||K^{\flat }-K^{\prime }{}^{\flat }||_{\alpha }^{2}\leq \frac{e^{\alpha T}-1%
}{\alpha }||K-K^{\prime }||_{\infty }^{2}
\end{equation*}%
so that, together with (\ref{key-est-a2}) (in which choose $\alpha
-\varepsilon C_{1}-\varepsilon ^{\prime }C_{2}=1$) 
\begin{eqnarray}
&&||\tilde{Y}-\tilde{Y}^{\prime }||_{\alpha }^{2}+||\tilde{Z}-\tilde{Z}%
^{\prime }||_{\alpha }^{2}+\beta ||\tilde{K}-\tilde{K}^{\prime }||_{\infty
}^{2}  \notag \\
&\leq &\left[ \left( 24TC_{1}^{2}+4C_{3}\right) \beta +\frac{2C_{1}}{%
\varepsilon }\right] \left( ||Y-Y^{\prime }||_{0}^{2}+||Z-Z^{\prime
}||_{0}^{2}\right)  \notag \\
&&+\left( \frac{24C_{2}^{2}}{\beta }T^{2}+\frac{2C_{2}}{\varepsilon ^{\prime
}\beta }\frac{e^{\alpha T}-1}{\alpha }\right) \beta ||K-K^{\prime
}||_{\infty }^{2}\text{.}  \label{contr-01}
\end{eqnarray}%
Choose $\varepsilon =8C_{1}$, $\varepsilon ^{\prime }=1$, $\alpha
=1+8C_{1}^{2}+C_{2}$ and $\beta =\frac{1}{16\left( 4TC_{1}^{2}+1\right) }$.
Then there is a number $C_{0}>0$ such that if $C_{2}T\leq C_{0}$, we have 
\begin{equation*}
12C_{2}^{2}T^{2}+C_{2}\frac{e^{\left( 1+8C_{1}^{2}+C_{2}\right) T}-1}{%
1+8C_{1}^{2}+C_{2}}\leq \frac{1}{48\left( 6TC_{1}^{2}+C_{3}\right) }
\end{equation*}%
so that%
\begin{equation*}
\frac{24C_{2}^{2}}{\beta }T^{2}+\frac{2C_{2}}{\varepsilon ^{\prime }\beta }%
\frac{e^{\alpha T}-1}{\alpha }\leq \frac{1}{2}\text{.}
\end{equation*}%
Hence%
\begin{equation}
||(\tilde{Y},\tilde{Z},\tilde{K})-(\tilde{Y}^{\prime },\tilde{Z}^{\prime },%
\tilde{K}^{\prime })||_{\alpha ,\beta }\leq \frac{1}{\sqrt{2}}%
||(Y,Z,K)-(Y^{\prime },Z^{\prime },K^{\prime })||_{\alpha ,\beta }\text{,}
\label{contr-a2}
\end{equation}%
so there is a fixed point $(Y,Z,K)$, which is clearly a solution according
to Proposition \ref{prop-03}.
\end{proof}

\subsection{Continuous dependence and uniqueness}

Here we study the continuous dependence result of the solution of this
reflected equation with respect to the parameters, which will lead to the
uniqueness of the solution immediately. First we consider following a priori
estimation.

\begin{proposition}
Under the same assumptions in Theorem \ref{ex-01}. Suppose $(Y,Z,K)$ to be
the solution of reflected BSDE(\ref{rbsde-ge}), then there exists a constant 
$C$ depending only on $C_{1}$ and $C_{2}T$, such that 
\begin{equation}
\mathbb{E}\left( \sup_{0\leq t\leq
T}Y_{t}^{2}+\int_{0}^{T}|Z_{s}|^{2}ds+K_{T}^{2}\right) \leq C\mathbb{E}%
\left( \xi ^{2}+\int_{0}^{T}(f_{t}^{0})^{2}dt+(\sup_{0\leq t\leq
T}S_{t}^{+})^{2}\right) \text{.}  \label{we-04}
\end{equation}
\end{proposition}

\begin{proof}
Applying It\^{o}'s formula to $Y_{t}^{2}$, and taking expectation, then 
\begin{eqnarray*}
\mathbb{E}[Y_{t}^{2}+\int_{t}^{T}|Z_{s}|^{2}ds &=&\mathbb{E}[\xi
^{2}+2\int_{t}^{T}Y_{s}f(s,Y_{s},Z_{s},K_{s})ds+2\int_{t}^{T}Y_{s}dK_{s}] \\
&=&\mathbb{E}[\xi
^{2}+2\int_{t}^{T}Y_{s}f(s,Y_{s},Z_{s},K_{s})ds+2\int_{t}^{T}S_{s}dK_{s}]
\end{eqnarray*}%
as $\int_{0}^{T}(Y_{s}-S_{s})dK_{s}=0$. By (\ref{lip-01}), we have 
\begin{eqnarray*}
\mathbb{E}[Y_{t}^{2}+\int_{t}^{T}|Z_{s}|^{2}ds] &\leq &\mathbb{E}\left( \xi
^{2}+2\int_{t}^{T}S_{s}dK_{s}\right)  \\
&&+2\mathbb{E}\int_{t}^{T}|Y_{s}|\left( \left\vert f_{s}^{0}\right\vert
+C_{1}\left( \left\vert Y_{s}\right\vert +\left\vert Z_{s}\right\vert
\right) +C_{2}\left\vert K_{s}\right\vert \right) ds \\
&\leq &\mathbb{E}\left( \xi ^{2}+\int_{t}^{T}\left\vert f_{s}^{0}\right\vert
^{2}ds+\alpha (\sup_{0\leq t\leq T}S_{t}^{+})^{2}\right) +\mathbb{E}%
\int_{t}^{T}|Z_{s}|^{2}ds \\
&&+\left( 2C_{1}+C_{1}^{2}+2\right) \mathbb{E}\int_{t}^{T}\left\vert
Y_{s}\right\vert ^{2}ds+(C_{2}^{2}T+\frac{1}{\alpha })\mathbb{E}K_{T}^{2}%
\text{.}
\end{eqnarray*}%
By using Gronwell's inequality, we get 
\begin{eqnarray*}
\mathbb{E}[Y_{t}^{2}] &\leq &e^{\left( 2C_{1}+C_{1}^{2}+2\right) T}\mathbb{E}%
[\left( \xi ^{2}+\int_{t}^{T}\left\vert f_{s}^{0}\right\vert ^{2}ds+\alpha
(\sup_{0\leq t\leq T}S_{t}^{+})^{2}\right)  \\
&&+(C_{2}^{2}T+\frac{1}{\alpha })e^{\left( 2C_{1}+C_{1}^{2}+2\right) T}%
\mathbb{E}K_{T}^{2}
\end{eqnarray*}%
and 
\begin{equation*}
\mathbb{E}\int_{t}^{T}|Z_{s}|^{2}ds\leq C_{T}\mathbb{E}\left[ \xi
^{2}+\int_{t}^{T}\left\vert f_{s}^{0}\right\vert ^{2}ds+\alpha (\sup_{0\leq
t\leq T}S_{t}^{+})^{2}+(C_{2}^{2}T+\frac{1}{\alpha })K_{T}^{2}\right] 
\end{equation*}%
where 
\begin{equation*}
C_{T}=\left( 2C_{1}+C_{1}^{2}+2\right) e^{\left( 2C_{1}+C_{1}^{2}+2\right)
T}T+1\text{.}
\end{equation*}%
Since 
\begin{equation*}
\text{ \ \ \ }K_{T}=Y_{0}-\xi
-\int_{0}^{T}f(s,Y_{s},Z_{s},K_{s})ds+\int_{0}^{T}Z_{s}dB_{s}\text{ \ }
\end{equation*}%
so 
\begin{eqnarray*}
\mathbb{E}K_{T}^{2} &\leq &C\mathbb{E}\left( \xi ^{2}+\int_{t}^{T}\left\vert
f_{s}^{0}\right\vert ^{2}ds\right) +8C_{1}^{2}\mathbb{E}\int_{t}^{T}\left%
\vert Y_{s}\right\vert ^{2}ds \\
&&+(8C_{1}^{2}+8)\mathbb{E}\int_{t}^{T}|Z_{s}|^{2}ds+8C_{2}^{2}T\mathbb{E}%
K_{T}^{2} \\
&\leq &C\mathbb{E}\left( \xi ^{2}+\int_{t}^{T}\left\vert
f_{s}^{0}\right\vert ^{2}ds+\alpha (\sup_{0\leq t\leq
T}S_{t}^{+})^{2}\right)  \\
&&+\left( C_{4}C_{2}^{2}T^{2}+4C_{2}^{2}T+\frac{C_{4}}{\alpha }\right) 
\mathbb{E}K_{T}^{2}
\end{eqnarray*}%
where $C_{4}=\left( 8C_{1}^{2}+(8C_{1}^{2}+8)\left(
2C_{1}+C_{1}^{2}+2\right) \right) e^{\left( 2C_{1}+C_{1}^{2}+2\right) T}$
which is independent of $C_{2}$, so that if $C_{2}T$ small enough, and
choosing $\alpha $ big enough, we can ensure that $%
C_{4}C_{2}^{2}T^{2}+4C_{2}^{2}T+\frac{C_{4}}{\alpha }<1$, so that 
\begin{equation*}
\mathbb{E}[Y_{t}^{2}+\int_{0}^{T}|Z_{s}|^{2}ds+K_{T}]\leq C\mathbb{E}[\xi
^{2}+\int_{0}^{T}(f_{t}^{0})^{2}dt+(\sup_{0\leq t\leq T}S_{t}^{+})^{2}],
\end{equation*}%
with some constant $C$. (\ref{we-04}) follows by the use of the Burkholder
inequality.
\end{proof}

Now we consider the following continuous dependence theorem

\begin{theorem}
Under the same assumptions in Theorem \ref{ex-01}. Suppose $%
(Y^{i},Z^{i},K^{i})$, $(i=1,2)$ to be the solution of reflected BSDE (\ref%
{rbsde-ge}) with parameters $(\xi ^{i},f^{i},S^{i})$, respectively. Set 
\begin{eqnarray*}
\bigtriangleup Y &=&Y^{1}-Y^{2}\text{, }\bigtriangleup Z=Z^{1}-Z^{2}\text{, }%
\bigtriangleup K=K^{1}-K^{2}\text{,} \\
\bigtriangleup \xi &=&\xi ^{1}-\xi ^{2}\text{, }\bigtriangleup f=f^{1}-f^{2}%
\text{, }\bigtriangleup S=S^{1}-S^{2}\text{.}
\end{eqnarray*}%
Then 
\begin{eqnarray*}
&&\mathbb{E}\left( \sup_{0\leq t\leq T}\left\vert \bigtriangleup
Y_{t}\right\vert ^{2}+\int_{0}^{T}|\bigtriangleup Z_{s}|^{2}ds+\left\vert
\bigtriangleup K_{T}\right\vert \right) \\
&\leq &C\mathbb{E}\left( \bigtriangleup \xi ^{2}+\int_{0}^{T}\left\vert
\bigtriangleup f(t,Y_{t}^{1},Z_{t}^{1},K_{t}^{1})\right\vert ^{2}dt\right) \\
&&+C\Psi _{T}^{\frac{1}{2}}\left[ \mathbb{E}(\sup_{0\leq t\leq T}\left\vert
\bigtriangleup S_{t}\right\vert )^{2}]\right] ^{\frac{1}{2}}
\end{eqnarray*}%
where $C$ depends only on $C_{1}$ and $C_{2}T$, and 
\begin{eqnarray*}
\Psi _{T} &=&\mathbb{E}\left( |\xi ^{1}|^{2}+|\xi ^{2}|^{2}\right) +\mathbb{E%
}\left[ (\sup_{0\leq t\leq T}(S_{t}^{1})^{+})^{2}+(\sup_{0\leq t\leq
T}(S_{t}^{2})^{+})^{2}\right] \\
&&+\mathbb{E}\int_{0}^{T}(\left\vert f^{1}(t,0,0,0)\right\vert
^{2}+\left\vert f^{2}(t,0,0,0)\right\vert ^{2})dt\text{.}
\end{eqnarray*}
\end{theorem}

\begin{proof}
Applying It\^{o}'s formula to $\left\vert \bigtriangleup Y_{t}\right\vert
^{2}$, then taking expectation, and using the fact that $\int_{t}^{T}(%
\bigtriangleup Y_{s}-\bigtriangleup S_{s})d(\bigtriangleup K_{s})\leq 0$, we
get 
\begin{eqnarray*}
&&\mathbb{E}\left( \left\vert \bigtriangleup Y_{t}\right\vert
^{2}+\int_{t}^{T}|\bigtriangleup Z_{s}|^{2}ds\right)  \\
&=&\mathbb{E}\left( \bigtriangleup \xi ^{2}+2\int_{t}^{T}\bigtriangleup
Y_{s}\bigtriangleup f(s,Y_{s}^{1},Z_{s}^{1},K_{s}^{1})ds\right)  \\
&&+2\mathbb{E}\left( \int_{t}^{T}\bigtriangleup
Y_{s}(f^{2}(s,Y_{s}^{1},Z_{s}^{1},K_{s}^{1})-f^{2}(s,Y_{s}^{2},Z_{s}^{2},K_{s}^{2}))ds\right) 
\\
&&+2\mathbb{E}\int_{t}^{T}\bigtriangleup S_{s}d(\bigtriangleup K_{s})\text{.}
\end{eqnarray*}%
By (\ref{lip-01}), we get 
\begin{eqnarray*}
\mathbb{E}\left\vert \bigtriangleup Y_{t}\right\vert ^{2} &\leq &\mathbb{E}%
\left( \bigtriangleup \xi ^{2}+\int_{t}^{T}\left\vert \bigtriangleup
f(s,Y_{s}^{1},Z_{s}^{1},K_{s}^{1})\right\vert ^{2}ds\right)  \\
&&+(2+2C_{1}+C_{1}^{2})\mathbb{E}\int_{t}^{T}\left\vert \bigtriangleup
Y_{s}\right\vert ^{2}ds+C_{2}^{2}T\mathbb{E}[\bigtriangleup K_{T}^{2}] \\
&&+\mathbb{E}\sup_{0\leq t\leq T}\left( \left\vert \bigtriangleup
S_{t}\right\vert (K_{T}^{1}+K_{T}^{2})\right) \text{.}
\end{eqnarray*}%
Now we are in the same situation as in the previous proposition, with
similar arguments, as the Lipschitz constant $C_{2}$ is small, the result
thus follows immediately.
\end{proof}

It follows from the continuous dependence result, the solution of reflected
BSDE (\ref{rbsde-ge})-(\ref{cons-02}) is unique.

\bigskip

\noindent {\textbf{Acknowledgements.}} This research was carried out while
the second author was visiting the Mathematics Institute, Oxford, and the
Oxford-Man Institute, partially supported by EPSRC\ grant EP/F029578/1 and
an LMS scheme 2 grant. The second author is partly supported by NSF Youth
Grant 10901154/A0110.

\bigskip

\bibliographystyle{amsplain}
\bibliography{RBSDE-Qian-Xu1}

\noindent {\small \textsc{Zhongmin Qian}, Mathematical Institute, University
of Oxford, Oxford OX1 3LB, England}

\noindent {\small {Email: \texttt{qianz@maths.ox.ac.uk}}}

\vskip0.2truecm

\noindent {\small \textsc{Mingyu Xu}, Key Laboratory of Random Complex
Structures and Data Science, Academy of Mathematics \& Systems Science, CAS,
Beijing }

\noindent {\small {Email: \texttt{xumy@amss.ac.cn}}}

\end{document}